\documentclass[11pt]{amsart}
\usepackage{amsfonts,latexsym,amsthm,amssymb,graphicx}
\usepackage[all]{xy}
\usepackage{hyperref}
\usepackage[usenames]{color} 
\usepackage{manfnt}

\DeclareFontFamily{OT1}{rsfs}{}
\DeclareFontShape{OT1}{rsfs}{n}{it}{<-> rsfs10}{}
\DeclareMathAlphabet{\mathscr}{OT1}{rsfs}{n}{it}

\newtheorem{theorem}{Theorem}[section]
\newtheorem{lemma}[theorem]{Lemma}
\newtheorem{corol}[theorem]{Corollary}

{\theoremstyle{definition} }
{\theoremstyle{remark} \newtheorem{remark}[theorem]{Remark}
\newtheorem{example}[theorem]{Example}}

\newcommand{\Pbb}{{\mathbb{P}}}
\newcommand{\Rbb}{{\mathbb{R}}}
\newcommand{\Zbb}{{\mathbb{Z}}}
\newcommand{\cL}{{\mathscr L}}
\newcommand{\ua}{\underline a}
\newcommand{\ue}{\underline e}
\newcommand{\um}{\underline m}
\newcommand{\uv}{\underline v}
\newcommand{\uX}{\underline X}

\DeclareMathOperator{\rk}{rk}
\DeclareMathOperator{\im}{im}
\DeclareMathOperator{\Vol}{Vol}
\DeclareMathOperator{\hVol}{\widehat \Vol}

\newcommand{\qede}{\hfill$\lrcorner$}

\title{
Multidegrees of monomial rational maps
}
\author{Paolo Aluffi}
\address{
Mathematics Department, 
Florida State University,
Tallahassee FL 32306, U.S.A.
}
\email{aluffi@math.fsu.edu}

\begin{document}

\begin{abstract}
We prove a formula for the multidegrees of a rational map defined by generalized 
monomials on a projective variety, in terms of integrals over an associated Newton
region. This formula leads to an expression of the multidegrees as volumes of
related polytopes, in the spirit of the classical Bernstein-Kouchnirenko theorem,
but extending the scope of these formulas to more general monomial maps.
We also determine a condition under which the multidegrees may be computed
in terms of the characteristic polynomial of an associated matrix.
\end{abstract}

\maketitle

%%%

\section{Introduction}\label{intro}

\subsection{}
Let $V\subseteq \Pbb^r$ be a projective variety, and $\varphi: V \dashrightarrow \Pbb^N$
a rational map. The {\em multidegrees\/} $\gamma_\ell$ of $\varphi$ are the coefficients of 
the class of the (closure of the) graph $\Gamma$ of this map in $\Pbb^r\times \Pbb^N$, 
to wit
\[
\gamma_\ell = h^{\dim V-\ell}\cdot H^\ell\cdot \Gamma\quad,
\]
where $h$, resp., $H$ is the pull-back of the hyperplane class in $\Pbb^r$, resp., 
$\Pbb^N$. The numbers $\gamma_\ell$ are obviously significant: for example,
$\gamma_r=0$ if and only if the general fiber of $\varphi$ is positive dimensional;
and if the general fiber consists of $D$ reduced points, then $\gamma_r$
is the product of $D$ and $\deg (\overline{\im \varphi})$. When $V=\Pbb^r$ and
$N=r$, $\varphi$ is a Cremona transformation if and only if $\gamma_r=1$. 
In general, $\gamma_0=\deg V$ and $\gamma_i=0$ if and only if 
$i>\dim(\im\varphi)$; if $V=\Pbb^r$, and for $i < \dim(\im\varphi)$, $\gamma_i$~may be 
interpreted as the degree of the closure of the image of a general $\Pbb^i\subseteq \Pbb^r$. 
We assemble the multidegrees into a polynomial
\[
\gamma_\varphi(t)=\gamma_0+\gamma_1 t+\gamma_2 t^2+\cdots
\]
of degree $\dim \im \varphi$. This polynomial does not depend on the dimension $r$ of 
the space containing $V$; it does depend on the hyperplane class of the embedding 
$V\hookrightarrow \Pbb^r$. We can in fact define a {\em multidegree class\/} 
(\S\ref{sec:multclass}) on $V$ as the push-forward to $V$ of $(\sum_{\ell\ge 0}H^\ell)
\cap [\Gamma]$; this is independent of any projective embedding of $V$, and our results 
will in fact deal with this class. In this introduction we will state the results for the 
multidegree polynomial, to remain closer to the more classical notion.

We consider rational maps $\varphi$ whose components are monomials $\mu_0,\dots, 
\mu_N$ in sections $s_j$ of line bundles $\cL_j$, $j=1,\dots,n$ on $V$, of course subject 
to the condition that all monomials are sections of the same line bundle $\cL$. For example,
$\varphi$ could be the restriction to $V$ of a rational map $\Pbb^r \dashrightarrow \Pbb^N$
defined by isobaric monomials in a collection of homogeneous polynomials. The hypersurfaces 
$X_j$ defined by $s_j$ on $V$ are required to satisfy a weak transversality hypothesis,
explained in \S\ref{monomialsetup}. For simplicity, the reader may assume that $V$ is
nonsingular and the $X_j$ form a simple normal crossing divisor, but less is needed:
see~\S\ref{monomialsetup} for an example of a much weaker sufficient condition.

We now state the result. 
The monomials $\mu_i = s_1^{m_{i1}}\cdots s_n^{m_{in}}$ determine lattice points 
$(m_{i1},\dots, m_{in})$ in $\Zbb^n\subseteq \Rbb^n$ (with coordinates $(a_1,\dots,
a_n)$). We move this set of points so
that one of them is at the origin, by setting $m'_{ij}=m_{ij}-m_{Nj}$ for $j=0,\dots, N$.
Notice that the lattice points $\um'_i=(m'_{i1},\dots, m'_{in})$ all lie on the subspace 
$d_1 a_1+\cdots + d_n a_n=0$, where $d_j=h^{\dim V-1}\cdot X_j$ is the degree 
of $X_j$ viewed as an algebraic set in $\Pbb^r$.

We denote by $N_\varphi$ the convex hull of the positive orthants translated at the lattice 
points~$\um'_i$; we call $N_\varphi$ the {\em Newton outer region\/} of $\varphi$.

\begin{theorem}\label{main1}
\begin{equation}\label{eq:main1}
\gamma_\varphi(t)=\int_{N_\varphi} \frac{n! X_1\cdots X_n\, t^n h^{\dim V+1}\, da_1\cdots da_n}
{(h+(a_1 X_1+\cdots +a_n X_n)t)^{n+1}}\quad.
\end{equation}
\end{theorem}

\begin{remark}
(i) The integral should be interpreted as follows. Perform the integral with $X_j$, $h$, and $t$
as parameters; the result is a rational function in these parameters. Replacing the 
parameters $X_j$ by the classes of the corresponding divisors in $V$ and $h$ by the
restriction of the hyperplane class from $\Pbb^r$ gives a 
{\em polynomial\/} in $t$; the coefficient of $t^\ell$ in this polynomial is a 
homogeneous polynomial of degree $\dim V$ in the classes $X_j$ and $h$. The statement
is that replacing the terms $h^{\dim V-\ell}\cdot X_{j_1}\cdots X_{j_\ell}$ in this polynomial 
with the corresponding intersection numbers determines the $\ell$-th multidegree~$\gamma_\ell$.

(ii) The convex region $N_\varphi$ depends on the choice of the pivoting monomial. It is a
consequence of the theorem that this choice does not affect the result of evaluating the
integral as specified in (i).
\qede\end{remark}

\begin{example}\label{tria}
Let $F_1,F_2,F_3$ be general homogeneous polynomials in $x_0$, $x_1$, $x_2$ of 
degrees $1$, $2$, $3$ respectively. Consider the rational map $\varphi: \Pbb^2 \dashrightarrow 
\Pbb^2$ given in components by $(x_0,x_1,x_2) \mapsto (F_2 F_3^2, F_1^2F_3^2, 
F_1^3 F_2 F_3)$.
According to Theorem~\ref{main1}, the multidegrees of $\varphi$ are the coefficients 
of $t^\ell$ in
\[
\int_{N_\varphi} \frac{n! X_1\cdots X_n\, t^n\, h^{\dim V+1}\,da_1\cdots da_n}
{(h+(a_1 X_1+\cdots +a_n X_n)t)^{n+1}}\quad,
\]
where $n=3$, $\dim V=2$, $X_1$, $X_2$, $X_3$ are the curves $F_1=0$, $F_2=0$, 
$F_3=0$, respectively, and $N_\varphi$ is the Newton outer region determined by the lattice points 
$(0,1,2)$, $(2,0,2)$, $(3,1,1)$ translated back to $A=(-3,0,1), B=(-1,-1,1), C=(0,0,0)$. 
That is, $N_\varphi$ is the region in $\Rbb^3$ extending from the triangle $ABC$ 
towards the three positive coordinate directions.
\begin{center}
\includegraphics[scale=.5]{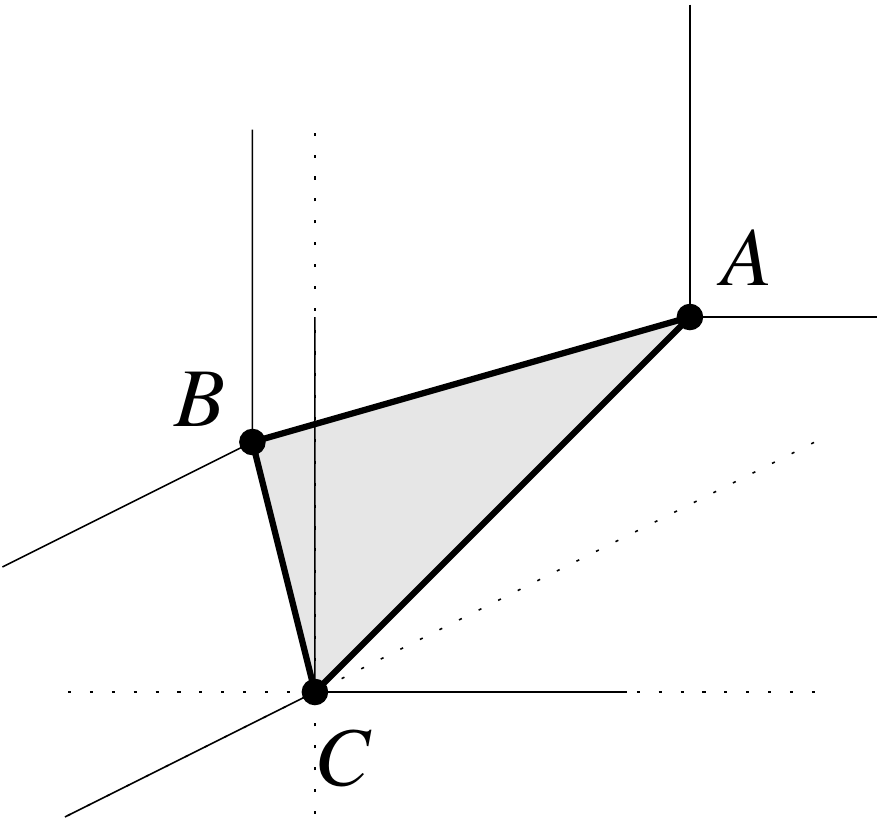}
\end{center}
The reader will easily verify that
\begin{multline*}
\int_{N_\varphi} \frac{6 X_1 X_2 X_3\, t^3\, h^3\, da_1 da_2 da_3}
{(h+(a_1 X_1+a_2 X_2 +a_3 X_3)t)^4} \\
=\frac{h^2(h^2+(-4X_1 +3X_3)ht+
(3X_1^2+3X_1X_2-5X_1X_3-X_2X_3+2X_3^2)\, t^2)}
{(h+(-3X_1+X_3)t)(h+(-X_1-X_2+X_3)t)}\quad.
\end{multline*}
As promised, the denominator disappears (canceling the extra factor at numerator)
after setting $X_1=h$, $X_2=2h$, $X_3=3h$; we get (as $h^2=1$)
\[
\gamma_\varphi(t)=1+5t+6t^2\quad,
\]
i.e., $\gamma_0=1$, $\gamma_1=5$, $\gamma_2=6$. This says that $\varphi$ is
generically $6$-to-$1$. Since the class of the monomials was $8h$ to begin with, 
we see that the base locus of $\varphi$ contributes $58$ to the intersection number
$(8h)^2$. 
\qede\end{example}

The integral appearing in Theorem~\ref{main1} may be computed from a decomposition
of the region $N_\varphi$ into simplices, including the positive coordinate directions
as possible vertices (at infinity). We will use notation as in \S2.2 of \cite{Scaiop}: in
particular, if a simplex $T$ has $s+1$ finite vertices, its {\em rank\/} is $\rk T=s$; and
the {\em volume\/} of $T$ is the normalized volume of the projection along its infinite 
directions. Also, we let $\deg T$ denote $h^{\dim V-\rk T} \cdot \prod X_j$, where
the (intersection) product is taken over the complement of the infinite directions in~$T$.
We denote by $\ua$ the vertex at infinity in the positive direction of the coordinate $a$.

\begin{theorem}\label{main2}
\begin{equation}\label{eq:main2}
\gamma_\varphi(t)=\sum_T \hVol(T)\,\deg(T)\, t^{\rk T}
\end{equation}
where the sum ranges over the full-dimensional simplices $T$ in a triangulation of 
$N_\varphi$, with vertices chosen among vertices of $N_\varphi$ and infinite coordinate 
vertices.
\end{theorem}

\begin{example}\label{tria2}
For the case illustrated in Example~\ref{tria}, a triangulation of $N_\varphi$ consists of
the simplices $T_1=C\ua_1\ua_2\ua_3$, $T_2=AC\ua_2\ua_3$, $T_3=AB\ua_1\ua_3$, 
$T_4=ABC\ua_1$. 
We have $\rk T_1=0$; $\rk T_2=\rk T_3=1$; $\rk T_4=2$; $\hVol(T_1)=1$, 
$\hVol(T_2)=3$, $\hVol(T_3)=\hVol(T_4)=1$ (as e.g., the projection of $T_4$ on the 
$a_2a_3$ plane is the triangle with vertices $(1,2)$, $(0,2)$, $(1,1)$); and 
$\deg T_1=h^{2-0}=1$, $\deg T_2=h^{2-1}\cdot X_1=1$, $\deg T_3=h^{2-1}\cdot X_2=2$, 
$\deg T_2=h^{2-2}\cdot X_2\cdot X_3 = 6$. According to Theorem~\ref{main2},
\[
\gamma_\varphi(t)=1+3\cdot 1\, t + 1\cdot 2\, t + 1\cdot 6\, t^2\quad,
\]
agreeing with the previous computation.
\qede\end{example}

For $X=\Pbb^r$ and $X_j=$ coordinate hyperplanes, the multidegrees of rational monomial 
maps are computed via toric methods by {\em mixed volumes\/} of Minkowsky sums of 
polytopes (see {\em e.g.,\/} \cite{MR2221122}, \S4, or \cite{Dolgachev}, \S3.5). In particular, 
the top degree may be expressed as the ordinary (normalized) euclidean volume of a convex 
polytope; this is a simple instance of the Bernstein-Kouchnirenko theorem. The leading coefficient 
of \eqref{eq:main2} reproduces this result in this case, and extends it to the more general 
monomial maps on projective varieties considered here, where toric techniques would not 
seem to be immediately 
applicable. Even in the particular case $X=\Pbb^r$, $X_j=$ coordinate hyperplanes,
it would be interesting to understand more fully why the ordinary volumes appearing in
Theorem~\ref{main2} compute the mixed volumes of the more classical formula.

In the case $N=r=n-1$, $N_\varphi$ has one finite face which (if non-degenerate) is an 
$(n-1)$-simplex in $\Rbb^n$. The map $\varphi$ is determined by the $n\times n$ 
matrix $M'_\varphi=(m'_{ij})$ whose rows $\um'_i$ consist of the translated lattice points, 
as above. (So one row of $M'_\varphi$ is $0$.) We say that $\varphi$ is 
{\em well-presented\/} if the following requirement on $M'_\varphi$ is
satisfied. Every choice of a set $I$ of indices $i_1,\dots,i_\ell$ determines 
a projection $\Rbb^n \to \Rbb^{n-\ell}$ along the coordinate directions $a_{i_1},\dots, 
a_{i_\ell}$. We require that the projection of the Newton outer region of $\varphi$ be
the Newton outer region determined by the projections of the rows $\um'_k$ for $k\notin I$. 
(We also require a condition on the signs of certain minors of $M_\varphi$; see
\S\ref{wp}.) For example, the standard Cremona transformation $(x_1:\dots: x_n) 
\mapsto (\frac 1{x_1}:\dots: \frac 1{x_n})$ trivially satisfies this condition. 

We prove that if $\varphi$ is well-presented, then the multidegree polynomial of 
$\varphi$ may be computed directly from the characteristic polynomial of the matrix 
$M'_\varphi$. The precise statement 
in the generality considered here is given in~Theorem~\ref{charpolmd};
for ordinary monomial rational maps $\varphi:\Pbb^{n-1}\dashrightarrow \Pbb^{n-1}$,
the result may be stated as follows. Let
\[
\alpha:\quad (x_1,\dots, x_{n-1}) \mapsto (x_1^{a_{11}} \cdots x_{n-1}^{a_{1, n-1}} , \dots,
x_1^{a_{n-1,1}} \cdots x_{n-1}^{a_{n-1, n-1}})
\]
be a morphism of tori, with $a_{ij}\in \Zbb$, inducing a rational map $\varphi:\Pbb^{n-1}
\dashrightarrow \Pbb^{n-1}$. Let $A=(a_{ij})$ be the matrix of exponents of $\alpha$, with 
characteristic polynomial $P_A(t)=\det(t\, I-A)$.

\begin{theorem}\label{pcchar}
If the rational map $\varphi:\Pbb^{n-1} \dashrightarrow \Pbb^{n-1}$ is well-presented, then
$\gamma_\varphi(t) = t^{n-1} P_A\left(\frac 1t\right)$.
\end{theorem}

It would be interesting to provide alternative characterizations of well-presented
monomial rational maps.\smallskip

{\em Acknowledgments.} The author is grateful to Igor Dolgachev for useful conversations,
particularly concerning the material in~\S\ref{pf3}.
The author's research is partially supported by a Simons collaboration grant. 

%%%

\section{Proof of Theorem~\ref{main1}}\label{pf1}
\subsection{}\label{sec:multclass}
Fulton-MacPherson intersection theory yields a direct relation between the multidegrees 
of a rational map $\varphi: \Pbb^r \dashrightarrow \Pbb^N$ and the degrees of the 
{\em Segre classes\/} of the base scheme of~$\varphi$: see e.g., \cite{MR1956868}, 
Proposition~3.1; \cite{MR2221122}, Proposition~5; or \cite{Dolgachev}, Proposition~2.3.1.
The case considered here requires the straightforward generalization of this relation
to the case of rational maps $\varphi: V \dashrightarrow \Pbb^N$, where $V$ is a 
subvariety of $\Pbb^r$.

Notation: Let $V$ be a closed subvariety (or subscheme) of $\Pbb^r$, $\cL$ a line bundle 
on $V$, and let $\varphi: V \dashrightarrow \Pbb^N$ be the rational map determined 
by a linear system in $H^0(V,\cL)$. Let $\Gamma\subseteq V\times \Pbb^N$ be the 
closure of the graph of $\varphi$, and let $G$ be its `shadow' in $V$:
\[
G:=\pi_*((1+H+H^2+ \cdots)\cap [\Gamma])\quad,
\]
where $H$ is the pull-back of the hyperplane class from the $\Pbb^N$ factor, and
$\pi: \Gamma \to V$ is the projection. This is the `multidegree class' mentioned in 
the introduction; the multidegrees of $\varphi$ are the degrees of the components of 
$G$, viewed as classes in $\Pbb^r$:
\[
\gamma_\varphi(t)=\int (1+ht+ht^2 + \cdots) \cap G=
\sum_{\ell\ge 0} (h^{\dim V-\ell}\cdot G_\ell)\, t^\ell \quad,
\]
where $h$ is the pull-back of the hyperplane class from $\Pbb^r$, and $G_\ell$ is 
the term of codimension $\ell$ in $G$. Thus, computing
the multidegree polynomial is reduced to computing the multidegree class $G$.

\begin{lemma}\label{multclass}
Let $i: S\subseteq V$ be the base scheme of the linear system defining $\varphi$. 
Then 
\[
G=c(\cL^*)^{-1}(([V]- i_* s(S,V))\otimes_V \cL^*)\quad.
\]
\end{lemma}
Here we are using the $\otimes$ notation introduced in~\S2 of~\cite{MR96d:14004}:
if $a$ is a class of codimension~$p$ in $A_*V$, then $a\otimes_V \cL^*$ denotes
$c(\cL^*)^{-p}\cap a$; the class $a\otimes_V\cL^*$ is defined for all $a\in A_*V$ by 
extending this prescription by linearity.

\begin{proof}
This argument is entirely analogous to the one given in the proof of Proposition~3.1
of~\cite{MR1956868}.
\end{proof}

\subsection{}\label{monomialsetup}
Now assume that $\varphi$ is monomial in the sense specified in~\S\ref{intro}: 
the linear system defining~$\varphi$ is generated by monomials $\mu_0,\dots, 
\mu_N$ in sections $s_j$ of line bundles $\cL_j$, $j=1,\dots,n$ on $V$;
if $\mu_i = s_1^{m_{i1}}\cdots s_n^{m_{in}}$, we assume $\cL_1^{\otimes m_{i1}}
\otimes\cdots\otimes \cL_n^{\otimes m_{in}}\cong \cL$ for all $i$.

We let $X_j$ denote the zero-scheme of $s_j$. The precise condition we put on
the hypersurfaces~$X_j$ is that monomial schemes defined with respect to
$X_1,\dots,X_n$ may be principalized by a sequence of blow-ups along 
codimension~$2$ monomial subschemes defined with respect to the proper transforms
of the $X_j$'s and the exceptional divisors in the blow-up sequence. By a theorem of 
R.~Goward (\cite{MR2165388}, Theorem~2), this condition holds if $V$ is nonsingular 
and the hypersurfaces $X_j$ form a divisor with simple normal crossings. As observed at 
the end of the introduction of~\cite{Scaiop}, the condition (and hence the results of this 
note) holds in greater generality: for example, $V$ can be an arbitrarily singular 
subvariety of $\Pbb^r$, so long as the $X_j$ are obtained as intersections of $V$ with 
components of a divisor with simple normal crossings in $\Pbb^r$ and that $V$ 
meets properly all strata of this divisor.

\subsection{}\label{transl}
In the situation described in~\S\ref{monomialsetup}, the base scheme $S$ of $\varphi$
is the subscheme of $V$ defined by the monomials $\mu_0,\dots,\mu_N$. This is
a monomial scheme in the sense adopted in~\cite{Scaiop}. As in~\cite{Scaiop},
we associate with the monomials $\mu_i = s_1^{m_{i1}}\cdots s_n^{m_{in}}$ the lattice 
points $(m_{i1},\dots,m_{in})$ in $\Rbb^n$, and the `Newton region' $N$ obtained as 
the (closure of the) complement in the positive orthant of the convex hull $N^c$ of the 
translations by $\mu_i$ of the positive orthants.
\begin{center}
\includegraphics[scale=.5]{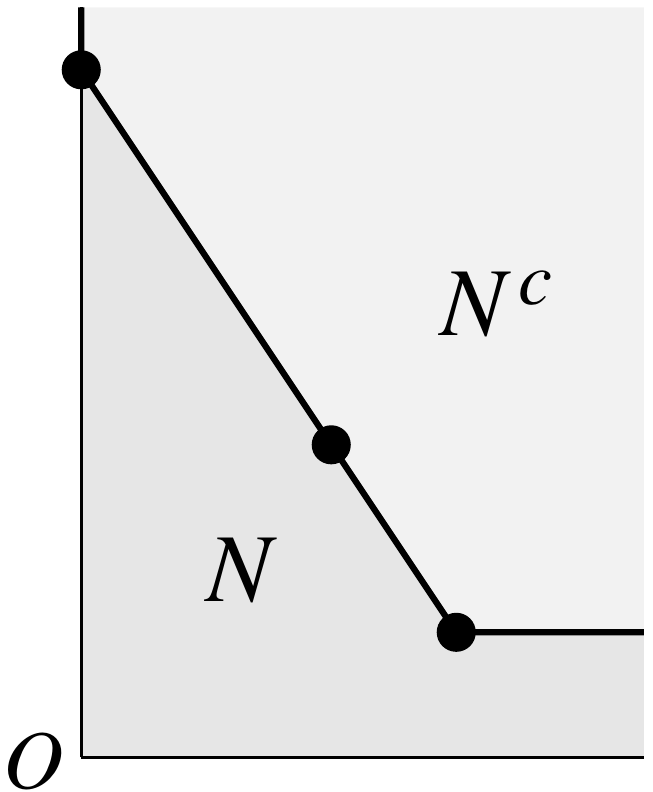}
\end{center}
If $d_j=h^{\dim V-1}\cdot X_j$, then for every $i$ we have $\sum d_j m_{ij} =
h^{\dim V-1}\cdot \sum m_{ij} X_j=h^{\dim V-1}\cdot c_1(\cL)\cap [V]=:d$. Therefore,
all the vertices $(m_{ij})$ belong to the hyperplane with equation $d_1 a_1+\cdots +
d_n a_n = d$ in $\Rbb^n$. We translate the vertices so that this hyperplane
goes through the origin, for example by subtracting the coordinates of one vertex. 
(The choice of this pivoting monomial will be irrelevant.)
\begin{center}
\includegraphics[scale=.5]{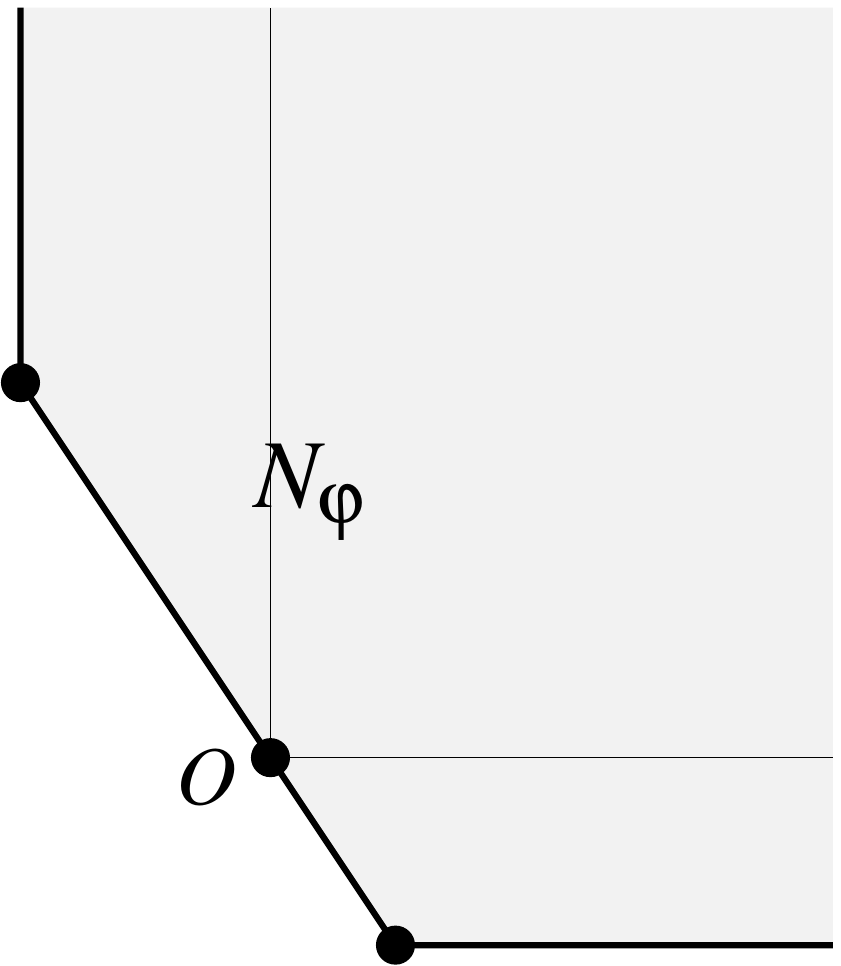}
\end{center}
The `Newton outer region' $N_\varphi$ is the corresponding translation of the 
region $N^c$.

\begin{lemma}\label{Segre}
With notation as above,
\[
[V]-i_* s(S,V)=\int_{N^c} \frac{n!\, X_1\cdots X_n\, da_1\cdots da_n}
{(1+a_1 X_1+\cdots +a_n X_n)^{n+1}}\quad.
\]
\end{lemma}

\begin{proof}
This follows immediately from the main theorem in~\cite{Scaiop}, which shows that the
Segre class is the integral of the same function over $N$. Since the integral over the
positive orthant is~$1$ (i.e., $[V]$), the integral over $N^c$ equals $[V]-i_* s(S,V)$
when viewed in $A_*V$, as stated.
\end{proof}

\subsection{}\label{sec:proof}
By Lemmas~\ref{multclass} and~\ref{Segre}, the multidegree class is given by
\[
G=c(\cL^*)^{-1}\left(\int_{N^c} \frac{n!\, X_1\cdots X_n\, da_1\cdots da_n}
{(1+a_1 X_1+\cdots +a_n X_n)^{n+1}}\otimes_V \cL^*\right)\quad.
\]
We can perform the $\otimes_V$ on the integrand. If $\mu = s_1^{m_{1}}\cdots 
s_n^{m_{n}}$ is the pivoting monomial, then $m_1 X_1+\cdots + m_n X_n$
represents $c_1(\cL)$ (once the $X_j$ are replaced with the homonymous 
cycles). Using the simple properties of $\otimes_V$ (cf.~\S2 of~\cite{MR96d:14004}),
\[
\frac{n!\, X_1\cdots X_n}
{(1+a_1 X_1+\cdots +a_n X_n)^{n+1}}\otimes_V \cL^*
=\frac{n!\, c(\cL^*)\cdot X_1\cdots X_n}
{(1+(a_1-m_1) X_1+\cdots +(a_n-m_n) X_n)^{n+1}}
\]
and hence
\begin{equation}\label{eq:premultclass}
G=\int_{N^c} \frac{n!\, X_1\cdots X_n\, da_1\cdots da_n}
{(1+(a_1-m_1) X_1+\cdots +(a_n-m_n) X_n)^{n+1}}
\end{equation}
Now, as $(a_1,\dots,a_n)$ ranges over $N^c$, the translated point 
$(a_1-m_1,\dots,a_n-m_n)$ ranges over $N_\varphi$. Therefore,

\begin{theorem}\label{primary}
With notation as above, the multidegree class of a monomial rational map 
$V\dashrightarrow \Pbb^n$ is given by
\begin{equation}\label{eq:multclass}
G=\int_{N_\varphi} \frac{n!\, X_1\cdots X_n\, da_1\cdots da_n}
{(1+a_1 X_1+\cdots + a_n X_n)^{n+1}}
\end{equation}
\end{theorem}

Theorem~\ref{primary} is the primary result. To complete the proof of
Theorem~\ref{main1}, it suffices to read the degree $h^{\dim V-\ell}\cdot G_\ell$
off the components of the multidegree class, where $G_\ell$ has codimension~$\ell$
in $V$. From~\eqref{eq:multclass}, inserting a dummy variable $u$ to keep track of
codimensions, we see that
\begin{equation}\label{multclasst}
G_0+G_1 u+ G_2 u^2 + \cdots =
\int_{N_\varphi} \frac{n!\, X_1\cdots X_n\, u^n\,da_1\cdots da_n}
{(1+(a_1 X_1+\cdots + a_n X_n) u)^{n+1}}\quad.
\end{equation}
Formally,
\begin{align*}
\gamma_\varphi(t) &=h^{\dim V}\cdot G_0+ h^{\dim V-1}\cdot G_1\, t
+ h^{\dim V-2}\cdot G_2\, t^2 +\cdots) \\
&= h^{\dim V}\left( G_0 + G_1 \frac th + G_2 \frac {t^2}{h^2} +\dots\right)
\end{align*}
Implementing this formal manipulation in~\eqref{multclasst} yields the integral
given in~\eqref{eq:main1}, concluding the proof of Theorem~\ref{main1}.

%%%

\section{Proof of Theorem~\ref{main2}}\label{pf2}

\subsection{}
Integrals such as the one appearing in~Theorem~\ref{primary} may be computed
from a triangulation of the region~$N_\varphi$. For us, a {\em generalized simplex\/} 
$T$ of {\em rank\/} $s$ in $\Rbb^n$ is the subset spanned by $s+1$ affinely independent
points $\uv_0,\dots, \uv_s$, and $n-s$ positive 
coordinate directions $\ua_{j_1},\dots, \ua_{j_{n-s}}$ (`infinite vertices'). Thus, 
\[
T = \left\{\sum_{i=0}^s \alpha_i \uv_i+ \sum_{k=1}^{n-s} \beta_k \ue_{j_k}\quad | \quad
\text{$\forall i,k: \alpha_i\ge 0, \beta_k\ge 0$, and $\sum_i \alpha_i=1$}\right\}
\]
where $\ue_1=(1,0,\dots, 0)$, \dots, $\ue_n=(0,\dots, 0, 1)$.
The (normalized) volume $\hVol(T)$ of a simplex is the normalized volume of the
simplex obtained by projecting $T$ along its infinite directions. The simplex may
degenerate when the vertices are affinely dependent; the volume of such an
`empty' simplex is $0$.

A simple calculus exercise yields the following result.

\begin{lemma}\label{simplexco}
If $T$ has finite vertices $\uv_i=(v_{i1},\dots,v_{is})$, $i=0,\dots,s$, and infinite vertices
$\ua_{j_k}$, $k=1,\dots,n-s$, then
\[
\int_T \frac{n!\, X_1\cdots X_n\, da_1\cdots da_n}
{(1+a_1 X_1+\cdots +a_n X_n)^{n+1}}
=\frac{\hVol(T)}
{\prod_{i=0}^s (1+v_{i 1}X_1+\cdots+ v_{i n}X_n)}\cdot
\frac{X_1\cdots X_n}{\prod_{k=1}^{n-s} X_{j_k}}\quad.
\]
\end{lemma}

\begin{proof}
\cite{Scaiop}, Lemma~2.5.
\end{proof}

\subsection{}
The point now is that if $(v_{i1},\dots, v_{in})$ is one of the translated monomials, i.e., one
of the vertices of $N_\varphi$, then after evaluating the $X_i$'s to the corresponding
classes,
\[
v_{i1}X_1+\cdots+ v_{i n}X_n = 0\quad:
\]
indeed, $v_{i 1}X_1+\cdots+ v_{i n}X_n$ is obtained by subtracting two classes 
both representing $c_1(\cL)$. Therefore, if $T$ is part of a triangulation of $N_\varphi$,
and the finite vertices of $T$ are vertices of~$N_\varphi$, then the contribution of $T$ to 
the multidegree class $G$ is simply (by Theorem~\ref{primary}) 
\[
\hVol(T)\cdot
\frac{X_1\cdots X_n}{\prod_{k=1}^{n-s} X_{j_k}}\quad.
\]
The factor
\[
X_T:=\frac{X_1\cdots X_n}{\prod_{k=1}^{n-s} X_{j_k}}
\]
for the generalized simplex $T$ is the product of the classes $X_j$ such that $\ua_j$
is {\em not\/} an infinite vertex of $T$. As a class in $A_*V$, $X_T$ has codimension
equal to the rank of $T$. Summarizing,

\begin{corol}
With notation as above, the multidegree class for $\varphi$ is given by
\begin{equation}\label{eq:simplices}
G=\sum_T \hVol(T) \cdot X_T
\end{equation}
where the sum is over the generalized simplices in a triangulation of $N_\varphi$, 
with finite vertices at vertices of $N_\varphi$.
\end{corol}

\begin{remark}
The region $N_\varphi$ may be viewed as the convex hull of the (finite) translated monomials
and of the (infinite) positive coordinate directions. As such, it always admits a triangulation
whose simplices have vertices among these points, cf.~\S2.2 
in~\cite{MR2743368}.
\qede\end{remark}

\subsection{}
If an embedding of $V$ in a projective space $\Pbb^r$ has been chosen, and $h$ is the
restriction of the hyperplane class, it is now natural to let the {\em degree\/} of $T$ be the 
intersection number $h^{\dim V -\rk T} \cdot X_T$. By~\eqref{eq:simplices}, we have
\[
\gamma_\ell = \sum_{\rk T=\ell} \hVol(T) \deg(T)\quad,
\]
where the sum is over (maximal dimension) simplices of fixed rank in a triangulation of
$N_\varphi$. In other words,
\[
\gamma_\varphi(t)= \sum_T \hVol(T) \deg (T)\, t^{\rk (T)}\quad,
\]
concluding the proof of~Theorem~\ref{main2}.

%%%

\section{Well-presented rational maps}\label{pf3}
\subsection{}
We now consider the $n\times (N+1)$-matrix $M_\varphi$ whose rows are the
vectors
\[
\um_i=(m_{i1}, \cdots, m_{in})
\]
determined by the exponents of the monomials defining the rational map~$\varphi$.
We aim at identifying a condition under which the multidegree polynomial of $\varphi$ 
may be obtained directly from this matrix. We specialize to the case $N=n-1$, 
and renumber the monomials from $1$ to $n$, so the matrix $M_\varphi:=
(m_{ij})_{\substack{i=1,\dots, n \\ j=1,\dots,n}}$ is square.
The motivating example is the case of dominant (ordinary) monomial maps $\Pbb^{n-1} 
\dashrightarrow \Pbb^{n-1}$, and monomial Cremona transformations in particular 
(cf.~\S3.5  of~\cite{Dolgachev}); we remind the reader that our context is more general,
in that the source need not be $\Pbb^{n-1}$ (or even $(n-1)$-dimensional) and the
monomials may be built on sections of line bundles, cf.~\S\S\ref{intro} and~\ref{pf1}.

\subsection{}\label{wp}
We say that
$\varphi$ is {\em well-presented\/} by $M=M_\varphi$ if the following condition holds.
For every subset $I=\{i_1,\dots,i_\ell\}\subseteq\{1,\dots, n\}$, we let $M^I$ be the 
matrix obtained by removing the $i$-th row and column of $M$ for all $i\in I$. 
We also let $\pi_I$ denote the projection $\Rbb_{\ge 0}^n \to \Rbb_{\ge 0}^{n-\ell}$ 
along the $i_1,\dots, i_\ell$ directions. We require that
\begin{itemize}
\item The projection $\pi_I(N_\varphi)$ of the Newton outer region of $\varphi$ equals 
the Newton outer region determined by the projections $\pi_I(\um_k)$ for $k\not\in I$; and
\item For all $I\subsetneq \{1,\dots,n\}$, the determinant of $M^I$ is either $0$ or has 
sign $(-1)^{n-1-|I|}$.
\end{itemize}

Roughly, these conditions say that the ordered simplex determined by the rows of the
matrix $M$ is in sufficiently general position with respect to the coordinate directions.

\begin{example}\label{cremo}
The standard Cremona transformation $\Pbb^2 \dashrightarrow \Pbb^2$ 
$(x_1:x_2:x_3) \mapsto (x_2 x_3: x_1 x_3: x_1 x_2)$ with matrix
\[
\begin{pmatrix}
0 & 1 & 1 \\
1 & 0 & 1 \\
1 & 1 & 0
\end{pmatrix}
\]
is well-presented.
\begin{center}
\includegraphics[scale=.5]{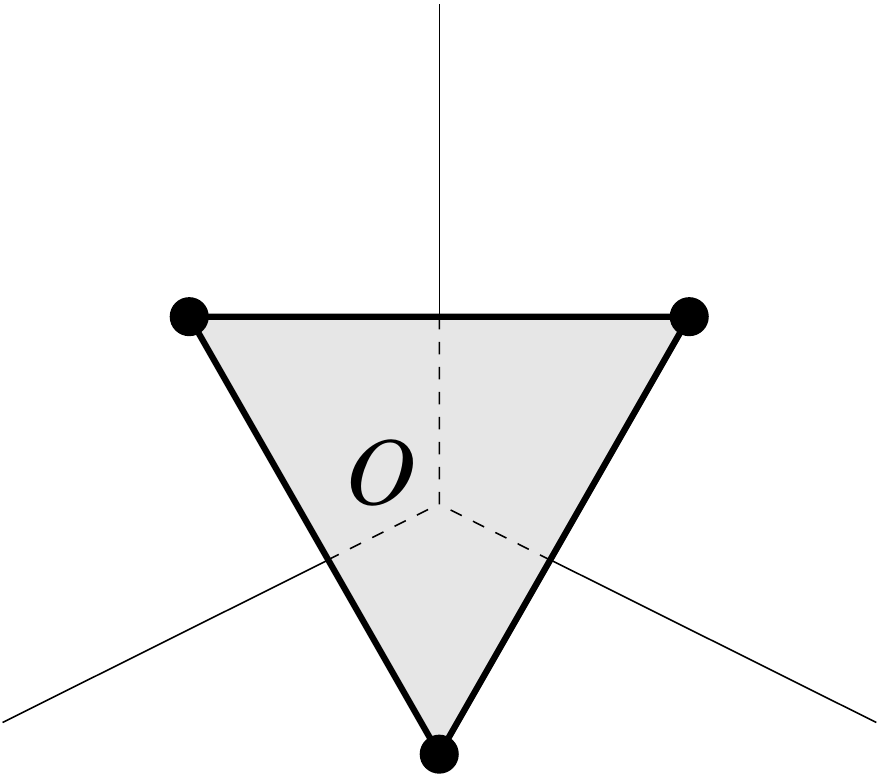}
\end{center}
With (for example) $I=\{3\}$, we see that the projection $(1,1)$ of the third row to the
horizontal plane is in the region determined by the projections $(0,1)$ and $(1,0)$ of the 
other rows:
\begin{center}
\includegraphics[scale=.5]{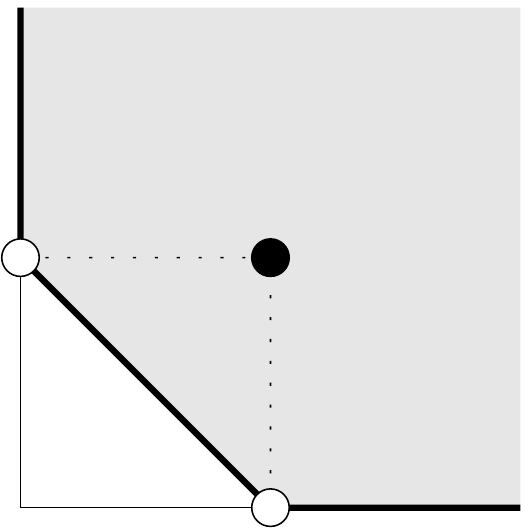}
\end{center}
Also, $\det M^{\{3\}}=\det \begin{pmatrix} 0 & 1 \\ 1 & 0 \end{pmatrix}=-1$, as 
required. The reader can verify that the conditions are satisfied for all choices of $I$.

On the other hand, the identity $\Pbb^2 \to \Pbb^2$ is {\em not\/} well-presented.
\begin{center}
\includegraphics[scale=.5]{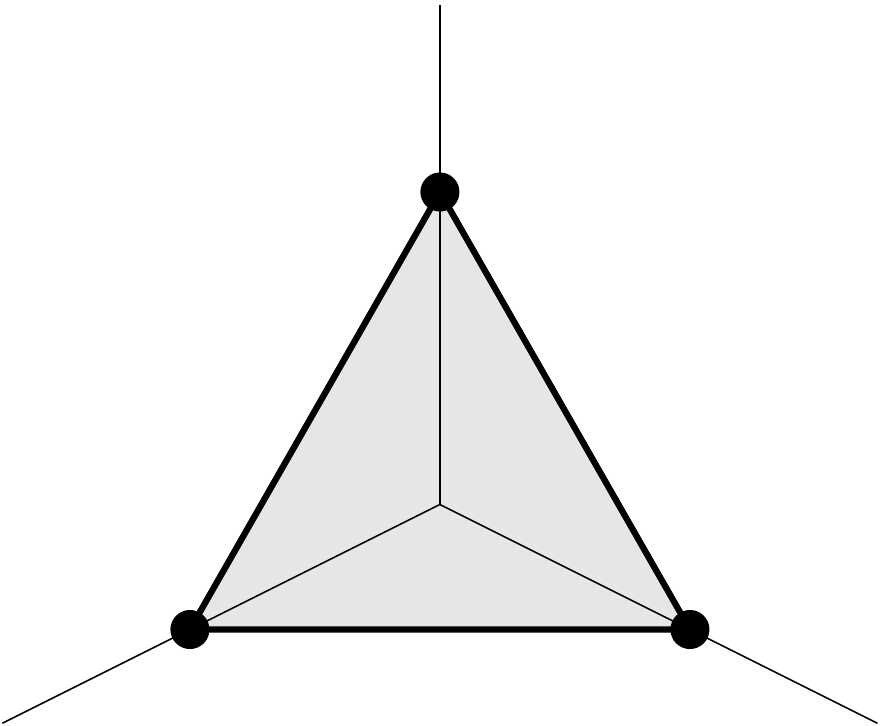}
\end{center}
For example, for $I=\{3\}$ we have the following projection:
\begin{center}
\includegraphics[scale=.5]{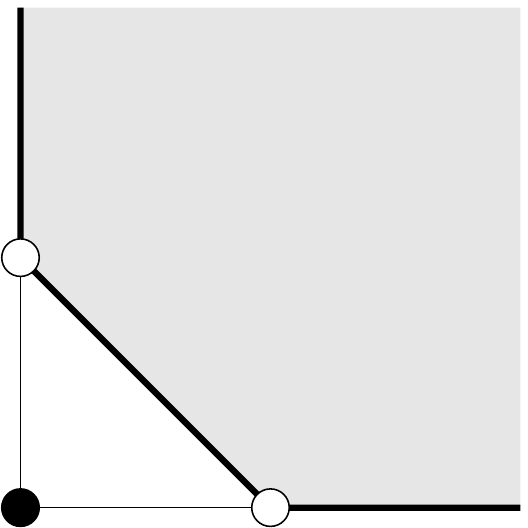}
\end{center}
The projection of the third row is not in the outer region determined by the other two.
Further, $\det M^{\{3\}}=1\ne (-1)^1$ in this case.
\qede\end{example}

\subsection{}
The main implication of the condition considered in~\S\ref{wp} is the following description
of the {\em Newton region\/} $N$ determined by a well-presented monomial map. This is the 
closure in the positive orthant of the complement of the outer region $N^c$ determined 
by the rows of~$M_\varphi$. (Recall that the region $N_\varphi$ is a translation of
$N^c$.) For example, for the Cremona transformation in
Example~\ref{cremo}, the Newton region consists of three infinite parallelepipeds along 
the coordinate axes, and one tetrahedron connecting them:
\begin{center}
\includegraphics[scale=.5]{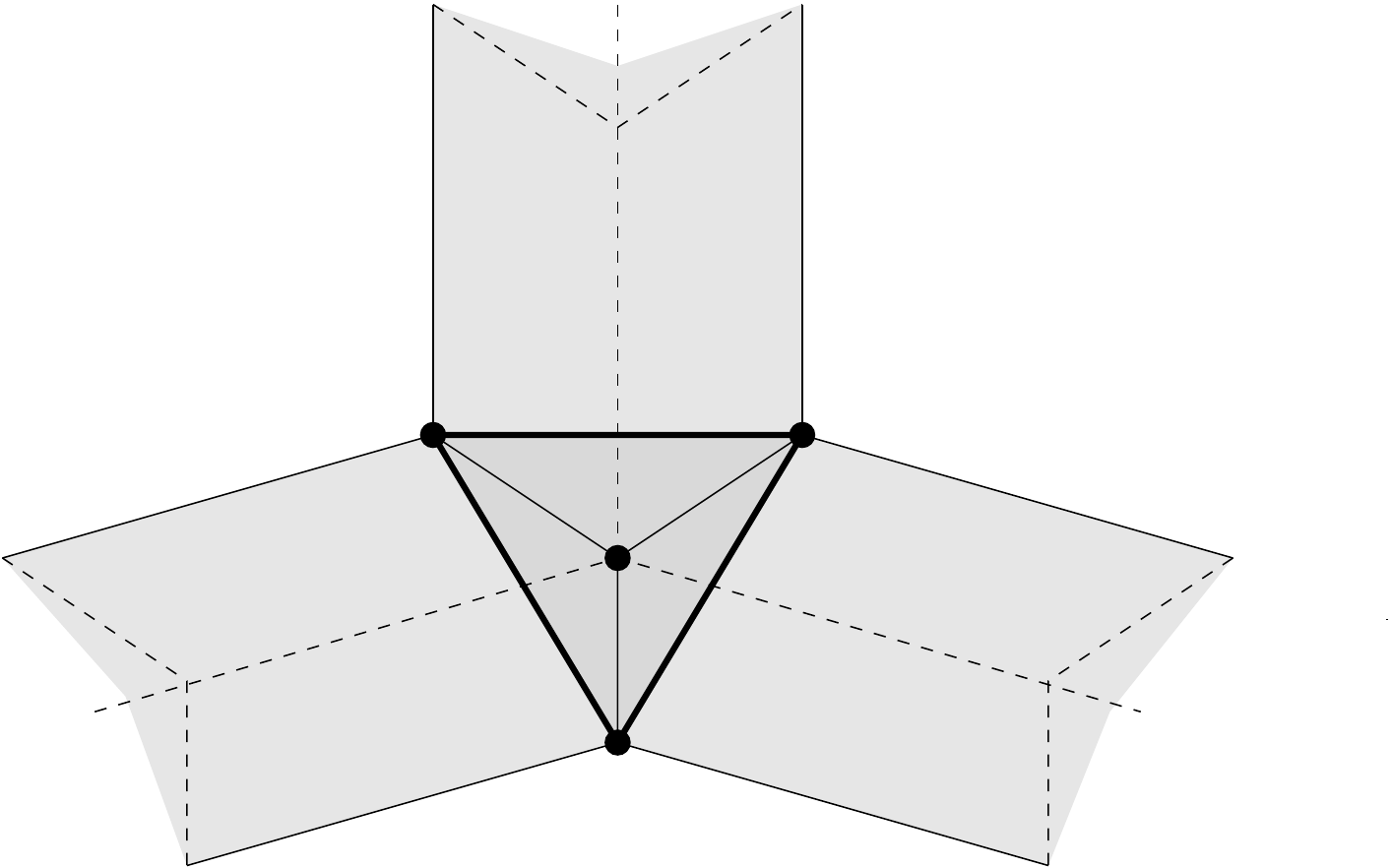}
\end{center}

\begin{lemma}\label{Newton}
If $\varphi$ is well-presented, then the corresponding Newton region is the (non-overlapping)
union over all proper subsets $I=\{i_1,\dots,i_\ell\}\subsetneq \{1,\dots,n\}$ of the generalized 
simplices with infinite vertices at $\ua_{i_1},\dots, \ua_{i_\ell}$ and finite vertices at the
origin and the rows $\um_k$ of $M_\varphi$ for $k\not\in I$.
\end{lemma}

The reader is invited to verify this statement on the Cremona example depicted above.
There are $7$ proper subsets of $\{1,2,3\}$; the tetrahedron in the middle corresponds
to $I=\emptyset$; the three infinite simplices to the singletons; and the three simplices 
corresponding to the remaining three subsets are empty.
More generally, the following picture (still for $n=3$) may help in visualizing the content of 
Lemma~\ref{Newton}:
\begin{center}
\includegraphics[scale=.5]{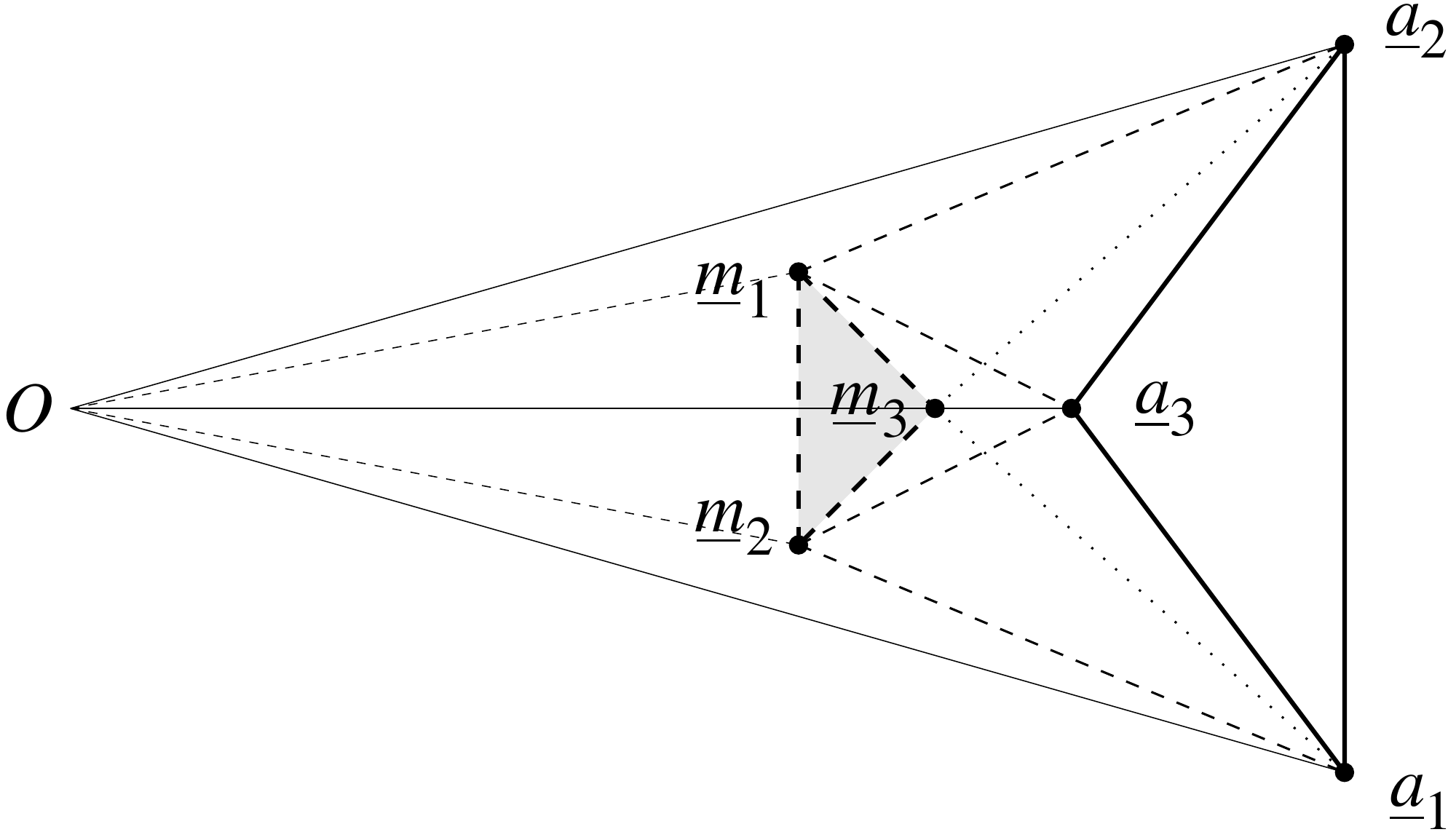}
\end{center}
The positive orthant is represented by the simplex $O\ua_1\ua_2\ua_3$, and $\ua_3$
points towards the reader. The triangle $\um_1\um_2\um_3$ is contained in the orthant,
and $\um_3$ points away from the reader. The region
between the triangles $\um_1\um_2\um_3$ and $\ua_1\ua_2\ua_3$ is the region $N^c$.
The region~$N$ is its complement, i.e., the union of the simplices $O\um_1\ua_2\ua_3$, 
$O\um_1\um_2\ua_3$, etc., as prescribed by the lemma. 
In the Cremona case of Example~\ref{cremo}, the vertices $\um_1$, $\um_2$, $\um_3$ are 
on the faces $O\ua_2\ua_3$, $O\ua_1\ua_3$, $O\ua_1\ua_2$, respectively, so three 
of the simplices are degenerate.

\begin{proof}
It is clear that $N$ may be decomposed as a union of simplices with vertices at 
the origin $O$, at some subset of infinite directions $\ua_i$, $i\in I\subseteq 
\{1,\dots, n\}$, and at a subset of the rows of $M$. Also, $I$ cannot consist of 
the whole $\{1,\dots,n\}$, since the generalized simplex with vertices $O,\ua_1,\dots,
\ua_n$ is the whole positive orthant. For $I=\{i_1,\dots, i_\ell\}$ properly contained
in $\{1,\dots, n\}$, 
we have to determine the set of $n-\ell$ rows $\um_k$ completing $O,\ua_{i_1},
\dots, \ua_{i_\ell}$ to a simplex $T$ contained in $N$. For this, we project along $I$.
For $i\in I$, $\pi_I(\um_i)$ is in the outer region determined by the projections 
$\pi_I(\um_k)$ for $k\not\in I$, by the first requirement listed in~\S\ref{wp}.
It follows that these latter $n-\ell$ rows must be the complementary set of vertices
of~$T$, and this is the assertion of the statement.
\end{proof}

\subsection{}
Lemma~\ref{Newton} and volume computations give the following result.
Recall that $G$ denotes the {\em multidegree class,\/} cf.~\S\ref{sec:multclass},
and that for $I\subseteq \{1,\dots,n\}$, $M^I$ denotes the matrix obtained by
removing the $i$-th row and column from $M$ for all $i\in I$.

\begin{corol}
If $\varphi$ is well-presented by the matrix $M=M_\varphi$, then
\begin{equation}\label{eq:prechar}
G=c(\cL^*)^{-1}\cap \left(\sum_{I\subseteq \{1,\dots,n\}} (-1)^{n-|I|} \det M^I 
\prod_{k\not\in I} X_k\right)\quad.
\end{equation}
\end{corol}

\begin{proof}
By~\eqref{eq:premultclass} in~\S\ref{sec:proof},
\[
G=\int_{N^c} \frac{n!\, X_1\cdots X_n\, da_1\cdots da_n}
{(1+(a_1-m_{n1}) X_1+\cdots +(a_n-m_{nn}) X_n)^{n+1}}
\]
where we have used $\um_n$ as pivot.
As $N$ is the complement of $N^c$ in the positive orthant, and the integral over
the orthant is $c(\cL^*)$,
the decomposition of $N$ obtained in Lemma~\ref{Newton} and the formula for
integrals over simplices (Lemma~\ref{simplexco}; remember that $O$ is a vertex
of each simplex) give
\[
G = c(\cL^*)^{-1}\cap \left(1-\sum_{I\subsetneq \{1,\dots,n\}} \hVol T_I X_{T_I}
\right)\quad,
\]
where $T_I$ is the simplex with vertices at the origin, $\ua_i$ for $i\in I$, and $\um_k$
for $k\not\in I$. Now $\hVol(T_I)$ is, by definition, the normalized volume of the 
simplex spanned by the finite vertices of $T_I$; thus, it equals $\pm$ the determinant
of the matrix obtained by replacing, for $i\in I$, the $i$-th row of $M$ with the
standard basis vector $\ue_i$. In other words, this volume equals $\pm \det M^I$,
and by the second requirement listed in \S\ref{wp} we have
\[
\hVol(T_I) = (-1)^{n-1-|I|} \det M^I
\]
for $I$ properly contained in $\{1,\dots,n\}$. Thus
\[
G = c(\cL^*)^{-1}\cap \left(1+\sum_{I\subsetneq \{1,\dots,n\}} (-1)^{n-|I|}\det M^I X_{T_I}
\right)\quad. 
\]
The statement follows, putting $\det M^I=1$ for $I=\{1,\dots,n\}$
(the `empty matrix').
\end{proof}

\subsection{}
Since the nonzero components of the class $G$ have codimension at most $n-1$, the 
right-hand side of~\eqref{eq:prechar} is necessarily a polynomial of degree less than $n$ 
in the $X_j$'s, once these are evaluated to the corresponding classes in $A_*V$. 
As we will see, this fact has a(n~even) more direct explanation.

We consider the matrix $M_\varphi(X)$ whose rows are the vectors
\[
(m_{i1}X_1, \cdots, m_{in}X_n)
\]
As usual, the $X_j$'s are considered as parameters at first, and will eventually be evaluated 
to the corresponding classes in $A_*V$. For the considerations which follow, we impose 
that 
\begin{equation}\label{eq:basrel}
m_{11}X_1 + m_{12} X_2 +\cdots + m_{1n} X_n = \cdots = 
m_{n1}X_1 + m_{n2} X_2 +\cdots + m_{nn} X_n\quad . 
\end{equation}
Once the $X_j$ are replaced with the corresponding classes $c_1(\cL_j)$ in $A_*V$, 
this common value is $c_1(\cL)$ by assumption. This amounts to the
statement that the column vector $\begin{pmatrix} 1 \\ \vdots \\ 1 \end{pmatrix}$ is
an eigenvector of $M_\varphi(X)$, with eigenvalue $c_1(\cL)$. Thus, the characteristic
polynomial of~$M_\varphi(X)$ has a factor of $(t-c_1(\cL))$:
\begin{equation}\label{eq:factordec}
\det(t\,I - M_\varphi(X)) = (t-c_1(\cL))\cdot Q(t)
\end{equation}
for a polynomial $Q(t)$ of degree $n-1$. On the other hand, the characteristic polynomial
of $M_\varphi(X)$ is precisely the term in parentheses appearing in~\eqref{eq:prechar}:

\begin{lemma}\label{ella}
\[
\det(t\,I-M_\varphi(X)) 
= \sum_{I\subseteq \{1,\dots,n\}} t^{|I|}(-1)^{n-|I|} \det M^I \prod_{k\not\in I} X_k\quad.
\]
\end{lemma}

\begin{proof}
This is elementary linear algebra.
\end{proof}

It follows that $G$ may be computed directly from the polynomial $Q(t)$:
\begin{corol}
If $\varphi$ is well-presented, and with notation as above,
$G=Q(1)$.
\end{corol}

\begin{proof}
By Lemma~\ref{ella}, this follows from~\eqref{eq:prechar} and~\eqref{eq:factordec}:
the factor $(t-c_1(\cL))$ equals $c(\cL^*)$ for $t=1$. 
\end{proof}

\subsection{}\label{sec:char}
We will now identify $Q(t)$ itself as a characteristic polynomial.
Let $M''_\varphi(X)$ be the matrix obtained by subtracting the last row of $M_\varphi(X)$
from the others, and discarding the last row and column:
\[
\begin{pmatrix}
m_{11} X_1 & m_{12} X_2 & m_{13} X_3  \\
m_{21} X_1 & m_{22} X_2 & m_{23} X_3  \\
m_{31} X_1 & m_{32} X_2 & m_{33} X_3 
\end{pmatrix}
\leadsto
\begin{pmatrix}
(m_{11}-m_{31}) X_1 & (m_{12}-m_{42}) X_2 \\
(m_{21}-m_{31}) X_1 & (m_{22}-m_{42}) X_2
\end{pmatrix}
\]

\begin{lemma}\label{Qaschar}
\[
Q(t)=\det(t\,I-M''_\varphi(X))\quad.
\]
\end{lemma}

\begin{proof}
This is also elementary linear algebra. Performing a change of basis from the standard
basis to
$\ue_1,\dots, \ue_{n-1}, \ue_1+\cdots + \ue_n$, $M_\varphi(X)$ is transformed into
\[
\begin{pmatrix}
1 & 0 & \cdots & 0 & -1 \\
0 & 1 & \cdots & 0 & -1 \\
\vdots & \vdots & \ddots & \vdots & \vdots \\
0 & 0 & \cdots & 1 & -1 \\
0 & 0 & \cdots & 0 & 1 \\
\end{pmatrix}
\cdot M_\varphi(X)\cdot
\begin{pmatrix}
1 & 0 & \cdots & 0 & 1 \\
0 & 1 & \cdots & 0 & 1 \\
\vdots & \vdots & \ddots & \vdots & \vdots \\
0 & 0 & \cdots & 1 & 1 \\
0 & 0 & \cdots & 0 & 1 \\
\end{pmatrix}=
\begin{pmatrix}
M''_\varphi(X) & 0 \\
* & c_1(\cL)
\end{pmatrix}
\]
hence
\[
\det(t\,I - M_\varphi(X)) = (t-c_1(\cL))\cdot \det(t\,I - M''_\varphi(X))\quad.
\]
The result follows by comparing with~\eqref{eq:factordec}.
\end{proof}

\begin{corol}\label{charact}
With notation as above, and assuming that $\varphi$ is well-presented, the multidegree 
class of $\varphi$ is obtained by evaluating the characteristic polynomial for the matrix 
$M''_\varphi(X)$ at $t=1$.
\end{corol}

\begin{remark}
Note that $X_n$ does not appear in $M''_\varphi(X)$, and hence in the expression for the
multidegree class obtained in Corollary~\ref{charact}. This is not too surprising, given
the redundancy built into the classes of the hypersurfaces $X_i$
(cf.~\eqref{eq:basrel}). Of course we could use the $i$-th row as pivot, and this would
yield an expression for the multidegree class in which $X_i$ does not appear.
\qede\end{remark}

\subsection{}\label{sec:cremona}
The matrix $M''_\varphi(X)$ has a compelling interpretation in the case of
rational maps $\varphi:\Pbb^{n-1}\dashrightarrow \Pbb^{n-1}$ whose components
are monomials in the homogeneous coordinates $x_1,\dots, x_n$. Every such map 
may be obtained by homogenizing a monomial morphisms of $(n-1)$-tori:
\begin{equation}\label{torimap}
\alpha:\quad (x_1,\dots, x_{n-1}) \mapsto (x_1^{a_{11}} \cdots x_{n-1}^{a_{1, n-1}} , \dots,
x_1^{a_{n-1,1}} \cdots x_{n-1}^{a_{n-1, n-1}})
\end{equation}
with $a_{ij}\in \Zbb$. A homogenization may be performed (for example) by
multiplying each monomial by a power of $x_n$ to obtain monomials of common 
degree~$0$, then multiplying each monomial by a common factor to obtain
nonnegative exponents throughout.

\begin{example}
Applying this procedure to the monomial map of tori
\[
(x_1,x_2,x_3) \mapsto (x_1^{-1} x_3, x_2^{-2}, x_2 x_3)
\]
gives the following rational map $\Pbb^3 \dashrightarrow \Pbb^3$:
\[
(x_1:x_2:x_3:x_4) \mapsto (x_1^{-1} x_3: x_2^{-2} x_4^2: x_2 x_3 x_4^{-2}:1)
= (x_2^2 x_3 x_4^2: x_1 x_4^4: x_1x_2^3 x_3:x_1 x_2^2 x_4^2)\quad.
\]
\end{example}

We say that $\alpha$ is well-presented if its homogenization is (in the sense 
specified in~\S\ref{wp}). We are ready to prove Theorem~\ref{pcchar} from the 
introduction:

\begin{theorem}\label{torimd}
Let $\alpha$ be a map of $(n-1)$-tori, let $A=(a_{ij})$ be the $(n-1)\times
(n-1)$ matrix of exponents, let $P_A(t)=\det(t\,I-A)$ be the characteristic
polynomial of $A$, and let $\gamma_\alpha(t)$ be the multidegree polynomial
for the corresponding rational map $\Pbb^{n-1}\dashrightarrow \Pbb^{n-1}$.
Assume that $\alpha$ is well-presented. Then
\[
\gamma_\alpha(t) = t^{n-1} P_A\left(\frac 1t\right)\quad.
\]
\end{theorem}

\begin{proof}
The matrix $h\cdot A$ obtained by multiplying each entry of $A$ by the hyperplane
class $h$ equals the matrix $M''_\varphi(X)$ for the rational map 
$\Pbb^{n-1} \dashrightarrow \Pbb^{n-1}$ obtained by homogenizing $\alpha$, and 
setting all $X_j$ to equal $h$. By Corollary~\ref{charact}, the multidegree class equals
$\det(I - h\cdot A)$, and the stated formula follows by formal manipulations.
\end{proof}

\begin{example}
The standard Cremona transformation corresponding to the map of tori
\[
(x_1,\dots, x_{n-1}) \mapsto (x_1^{-1},\dots, x_{n-1}^{-1})
\]
is well-presented (Example~\ref{cremo}). The exponent matrix is 
\[
A=\begin{pmatrix}
-1 & \cdots & 0 \\
\vdots & \ddots & \vdots \\
0 & \cdots & -1
\end{pmatrix}
\]
hence $P_A(t)=(t+1)^{n-1}$. According to Theorem~\ref{torimd}, its multidegree
polynomial is $\gamma_\alpha(t)=(1+t)^{n-1}$. Therefore its multidegrees are
$\gamma_\ell = \binom{n-1}\ell$ (cf.~\cite{MR2221122}, Theorem 2, 
and~\cite{Dolgachev},~\S3.4).
\qede\end{example}

\subsection{}
For more general well-presented monomial rational maps $V\dashrightarrow \Pbb^{n-1}$ 
based on $n$ hypersurfaces $X_1,\dots, X_n$, consider (as in \S\ref{pf1} and~\ref{pf2})
the Newton outer region $N_\varphi$, whose vertices are the rows of the $(n\times n)$
matrix $M'_\varphi(X)$ of translations $(\um_j-\um_n)\cdot \uX$:
\[
\begin{pmatrix}
m_{11} X_1 & m_{12} X_2 & m_{13} X_3  \\
m_{21} X_1 & m_{22} X_2 & m_{23} X_3  \\
m_{31} X_1 & m_{32} X_2 & m_{33} X_3  
\end{pmatrix}
\leadsto
\begin{pmatrix}
(m_{11}-m_{31}) X_1 & (m_{12}-m_{32}) X_2 & (m_{13}-m_{33}) X_3 \\
(m_{21}-m_{31}) X_1 & (m_{22}-m_{32}) X_2 & (m_{23}-m_{33}) X_3 \\
0 & 0 & 0 
\end{pmatrix}
\]
(Of course any row can serve as pivot.)

\begin{theorem}\label{charpolmd}
Let $\varphi$ be well-presented, and let $M'_\varphi(X)$ be as above, with characteristic
polynomial $P_{M'_\varphi(X)}(t)=\det(t\, I - M'_\varphi(X))$. Then the predegree polynomial 
of $\varphi$ is given by
\begin{equation}\label{eq:chargen}
\gamma_\varphi(t) = h^{\dim V-n} t^n P_{M'_\varphi(X)}\left(\frac ht\right)\quad.
\end{equation}
\end{theorem}

Identity \eqref{eq:chargen} should be interpreted by computing the right-hand side formally; 
this yields an expression in $t$ whose coefficients are homogeneous polynomials of degree 
$\dim V$ in the variables $h,X_1,\dots, X_n$. The statement is that evaluating the products as 
intersection products of the corresponding classes in $V$ gives the multidegree polynomial 
of $\varphi$.

\begin{proof}
The characteristic polynomial of $M'_\varphi(X)$ equals $t\, P_{M''_\varphi(X)}(t)$, and 
it follows from Lemma~\ref{Qaschar} and Corollary~\ref{charact} that
\[
P_{M'_\varphi(X)}(t)= t^n G_0 + t^{n-1} G_1 + \cdots + t G_{n-1}\quad,
\]
where $G_\ell$ is the term of codimension~$\ell$ in $G$. Since
\[
\gamma_\varphi(t) = h^{\dim V}\cdot G_0 + h^{\dim V-1}\cdot G_1 t+\cdots
\]
the stated formula follows by formal manipulations of these expressions.
\end{proof}

\begin{example}\label{tria3}
Returning to the case in Example~\ref{tria}, $\varphi$ is well-presented, as the reader
can easily verify. The matrices $M_\varphi$, $M'_\varphi(X)$ are
\[
\begin{pmatrix}
0 & 1 & 2 \\
2 & 0 & 2 \\
3 & 1 & 1
\end{pmatrix}
\quad,\quad
\begin{pmatrix}
-3X_1 & 0 & X_3 \\
-X_1 & -X_2 & X_3 \\
0 & 0 & 0
\end{pmatrix}
\quad.
\]
We have $\det(t\, I-M'_\varphi(X))=t(t+3X_1)(t+X_2)$. According to 
Theorem~\ref{charpolmd}, 
\[
\gamma_\varphi(t) = h^{2-3} t^3 \frac ht\left(\frac ht+3X_1\right) \left(\frac ht+X_2\right)
=h^2 + (3 h\cdot X_1+h\cdot X_2)\,t+3X_1 X_2 \,t^2
\]
Since $X_1=h$ and $\deg X_2=2h$ in this example, this recovers the result
$\gamma_\varphi(t)=1+5\,t+6\,t^2$ for the multidegree polynomial of $\varphi$, 
in agreement with the computations performed in~Examples~\ref{tria} and~\ref{tria2}.
\qede\end{example}

%%%

%%%

\end{document}